\documentclass[a4paper,11pt]{amsart}
\usepackage[a4paper]{geometry}

\usepackage{amsmath,amsthm,mdframed}
\usepackage{amssymb,amsfonts}
\usepackage[hidelinks]{hyperref}
\usepackage{indentfirst}
\usepackage{tikz}
\usepackage{epigraph}
\usepackage{enumitem} 
\usepackage{graphicx}
\usepackage{calligra}
\usepackage{stmaryrd}
\usepackage{chngpage}
\usepackage{tikz-cd}
\usepackage{mathtools}
\usepackage{relsize}
\usepackage{enumitem} 
\usepackage{mathrsfs} 
\usepackage{mathbbol}
 
\DeclareSymbolFont{bbold}{U}{bbold}{m}{n}
\DeclareSymbolFontAlphabet{\mathbbm}{bbold}

\theoremstyle{plain}
	\newtheorem{theorem*}{Theorem}
	\newtheorem{theorem}{Theorem}[section]
	
\numberwithin{equation}{section}

	\newtheorem{proposition}[theorem]{Proposition}
	\newtheorem{lemma}[theorem]{Lemma}
	\newtheorem{corollary}[theorem]{Corollary}
	\newtheorem{claim}{Claim}[theorem]

\theoremstyle{definition}
	\newtheorem{definition}[theorem]{Definition}
	\newtheorem{notation}[theorem]{Notation}

	\newtheorem*{acknow}{Acknowledgements}
\theoremstyle{remark}
	\newtheorem{remark}[theorem]{Remark}

\numberwithin{equation}{section}

\DeclareMathOperator{\Ult}{Ult}

\DeclareMathOperator{\rank}{rank}
\DeclareMathOperator{\Ext}{Ext}
\DeclareMathOperator{\dom}{dom}
\DeclareMathOperator{\ran}{ran}
\DeclareMathOperator{\Hull}{Hull}
\DeclareMathOperator{\trace}{tr}
\DeclareMathOperator{\otp}{otp}

\newcommand{\vin}{\rotatebox[origin=c]{90}{$\in$}}
\newcommand{\Ord}{Ord}

\newcommand{\bbP}{\mathbb{P}}

\newcommand{\iterat}{{\mathbf{It}}}
\newcommand{\kl}{\lessdot}
\newcommand{\pow}{\mathscr{P}}
\newcommand{\NS}{\mathsf{NS}_{\omega_1}}

\renewcommand{\phi}{\varphi}
\newcommand{\theforcingg}{ \mathbb{P}^{c\text{-}c}}
\newcommand{\theforcing}[2]{ \mathbb{P}^{c\text{-}c}(#1,#2) }

\title[{Increasing the second uniform indiscernible by strongly ssp forcing}]
{Increasing the second uniform indiscernible by strongly ssp forcing}

\date{\today}

\author[BdB]{Ben De Bondt}
\address[Ben De Bondt]
{Institut de Math\'ematiques de Jussieu (IMJ-PRG)\\
Universit\'e Paris Cit\'e\\
B\^atiment Sophie Germain\\
8 Place Aur\'elie Nemours \\ 75013 Paris, France}
\email{ben.de-bondt@imj-prg.fr}
\urladdr{https://perso.imj-prg.fr/ben-debondt/}

\author[BV]{Boban Veli\v{c}kovi\'c}
\address[Boban Veli\v{c}kovi\'c]
{Institut de Math\'ematiques de Jussieu (IMJ-PRG)\\
Universit\'e Paris Cit\'e\\
B\^atiment Sophie Germain\\
8 Place Aur\'elie Nemours \\ 75013 Paris, France}
\email{boban.velickovic@imj-prg.fr}
\urladdr{https://webusers.imj-prg.fr/~boban.velickovic/}
	
\keywords{Stationary Set Preserving Forcing,  Side Conditions, Second Uniform Indiscernible, Games.}

\subjclass[2020]{03E57, 03E50, 03E05}

\begin{document}
\maketitle
\setcounter{tocdepth}{1}
\begin{abstract} 
We introduce a new and natural stationary set preserving forcing  $\theforcing{\lambda}{\mu}$ that (under $\NS$ precipitous +
existence of $H_{\theta}^\#$ for a sufficiently large regular $\theta$) increases the second uniform indiscernible $\mathbf{u}_2$ beyond some given ordinal $\lambda.$ The forcing $\theforcingg$ shares this property with forcings defined in \cite{claverie_schindler} and \cite{ketlarzap}. As a main tool we use certain natural open two player games which are of independent interest, viz.\  the capturing games $\mathbf{G}^{cap}_M(X)$ and the catching-capturing games $\mathbf{G}^{c\text{-}c}_M(X)$. In particular, these games are used to isolate a special family of countable elementary submodels $M \prec H_\theta$ that occur as side conditions in $\theforcingg$ and thus allow to control the forcing in a strong way.
\end{abstract}

\section{Introduction}\label{S.Intro}
To date, the study of forcing axioms is one of the foremost lines in set theoretic research. Reflecting upon the nature of
this phenomenon, one could argue that the true reason for the tremendous success of a forcing axiom as $\mathsf{PFA}$ is largely twofold. On the one hand, proper forcings are sufficiently well-behaved to allow a very fruitful iteration theory to be developed. Secondly, and just as crucially, the class of proper forcings is sufficiently rich to allow for a variety of very different proper forcings, giving in turn rise to many applications. A versatile technique for constructing such proper forcings, allowing them to be tailored to perform various different tasks, was introduced by Todorcevic (\cite{todorcevic_proper_forcing_axiom} and \cite{todorcevic_directed_sets}, see also \cite{todorcevic_partitions_topology}, \cite{todorcevic_irredundant_sets}, \ldots) in the 1980’s and proceeds by including in a forcing condition finitely many working parts, together with a finite chain of countable elementary substructures of~$H_\theta.$ The behaviour of such forcing is determined by carefully fixing the allowed configurations of working parts and models occuring in the conditions. One can contrast this with the situation for stationary set preserving  (from now on: ssp) forcings. The satisfactory iteration theory of semiproper forcings makes up for the lack of a general iteration theory for ssp forcings and allows to prove the consistency relative to a supercompact cardinal of the ssp-forcing axiom $\mathsf{MM}.$ Furthermore, the power of Woodin’s $\mathbb{P}_{\max}$-forcing technique, in conjunction with the recent results by Asper\'o and Schindler (\cite{aspero_schindler}) rephrasing (under suitable conditions) Woodin's axiom $(\ast)$ as a bounded forcing axiom, suggest many possible applications of ssp forcing strictly beyond proper forcing. However, a technique for constructing ssp forcings just as versatile as Todorcevic’s side condition technique has largely remained absent, even though the $\mathcal{L}$-forcing technique developed by Jensen (\cite{jensen_Lforcing}) and further developed in \cite{claverie_schindler},\cite{doebler_schindler} and ultimately in \cite{aspero_schindler}, may be considered a candidate. In this article we exhibit one rudimentary instance of an alternative approach towards constructing specific ssp forcings that stays much closer to the side condition method for proper forcings. More precisely, we consider a reasonably simple forcing $\theforcingg$ that is almost, but not entirely, a pure side condition forcing. The twists are the following: the side conditions are (morally) allowed to grow through the forcing procedure and working parts are included in the forcing conditions with the sole purpose of guiding this growth of the models. To guarantee that these models still fulfill their task as side conditions (and thus make the resulting forcing ssp) they can however not be allowed to absorb any new countable ordinals. Consequently, the crucial idea in this prima facie quite flexible adaptation of the proper-style side condition method is that only models $M \prec H_\theta$  exhibiting the right kind of growth behaviour are admissable as side conditions. The gain for the particular forcing $\theforcing{\lambda}{\mu}$ is that instead of just adding any sequence $(M^G_\alpha : \alpha < \omega_1)$ of countable models with union equal to $H_\theta$ (as a trivial pure side condition forcing would do), it will add a special such sequence that transitive collapses to a generic iteration with $\omega_1$-th iterate of the form $H_\lambda.$ Moreover under the existence of $H_\lambda^\#$, it can be arranged that the first model in this generic iteration is iterable. Recall that under the presence of sharps for reals, all $\omega_1$-th iterates $ \mathscr{M}_{\omega_1}$ occuring in some generic iteration $\iterat = (\mathscr{M}_\beta, j_{\alpha\beta},G_\alpha : \alpha < \beta \leq \omega_1)$ of a countable transitive iterable model $\mathscr{M}_0$ are, by standard boundedness arguments, bounded in rank by the second uniform indiscernible ordinal ${\bf u}_2$. 
Thus, it makes sense to consider the forcing $\theforcing{\lambda}{\mu}$ within the light of the still poorly understood problem of controlling the size of ${\bf u}_2$ by forcing. This problem has a significant history, starting with \cite{steel_wesep} where Steele and Van Wesep force over suitable determinacy models
to obtain the consistency of $\mathsf{ZFC}$ + $\NS$ saturated + ${\bf u}_2=\omega_2$. Later, Woodin showed \cite[Theorem 3.17]{woodin} that in fact ${\bf u}_2 = \omega_2$ follows from $\NS$ saturated + $\pow(\omega_1)^\#$ exists. Thus, ${\bf u}_2 = \omega_2$ can be forced over a $\mathsf{ZFC}$-model from a Woodin cardinal with a measurable above it using the forcing introduced by Shelah to make $\NS$ saturated \cite[Theorem 2.64]{woodin}.  
In \cite{steel_welch}, Steel and Welch showed that
$(\forall x \in \mathbb{R})\ x^\# \text{ exists } + \text{ }{\bf u}_2 = \omega_2$
implies the existence of an inner model with a strong cardinal. The above premise also follows from the theory $\mathsf{BMM}$ + $\NS$ precipitous, whose precise consistency strength is itself unclear.
That  $\mathsf{BMM}$ + $\NS$ precipitous implies ${\bf u}_2 = \omega_2$ is in turn a consequence of the existence of ssp forcings such as $\theforcingg$, which, working under the theory $\mathsf{ZFC} + \NS$ precipitous + existence of sufficient sharps, increase the value of $\mathbf{u}_2$ beyond some arbitrarily high ordinal.
So far, two different such ssp forcings had been described in the literature, namely a Namba-like forcing in \cite{ketlarzap} and an $\mathcal{L}$-forcing in \cite{claverie_schindler}. It is worth to note that both these two forcings, as well as the forcing $\theforcingg$ described here, use the same argument to affect the value of $\mathbf{u}_2$ which goes back to Woodin's argument \cite[Theorem 3.17]{woodin}, namely they all proceed through adding a countable transitive iterable model $\mathscr{M}$ together with a suitable length-$\omega_1$ generic iteration that blows it up to arbitrarily large rank. Also in some other aspects all three forcings behave very much alike. In particular, all three share with Namba forcing the property that they give some large regular cardinal cofinality $\omega$. We believe that a main advantage of the construction as presented here is that the inclusion of suitable elementary substructures into the forcing conditions elucidates the proof of preservation of stationarity of ground-model stationary subsets of $\omega_1.$ Indeed, the argument for proving that $\theforcingg$ is ssp becomes entirely standard once the right class of models to be used as side conditions has been identified and the existence of projective stationary many such models has been proved. This additionally allows us to abstractly define a property of forcings that we call \emph{strongly ssp} and is reminiscent of the similar notion of strong properness that is commonly satisfied by proper side condition forcings. Strongly ssp forcings enjoy additional strong properties beyond being merely ssp such as not adding any fresh functions $f: \omega_1 \to V.$ The forcing $\theforcingg$ is strongly ssp and therefore must add the many new reals that thin out the uniform indiscernibles, without adding any fresh subsets of $\omega_1.$

\section{Preliminaries}\label{S.prelim}
We gather here several notions and results that will be needed below.

\subsection{Iterable structures}

$\NS$ denotes the non-stationary ideal on $\omega_1,$ $\NS^+$ is the set of stationary subsets of $\omega_1.$ We also use $\NS^+$ as a shorthand for the partially ordered set $\displaystyle (\NS^+, \subseteq)$, that is, whenever $\NS^+$ is without further indication treated as a poset, we have silently endowed it with the partial order $S \leq T \Leftrightarrow S \subseteq T.$ 

$\mathsf{ZFC}^-$ is the theory obtained from $\mathsf{ZF}$ by removing the powerset axiom and adding the collection axiom and the axiom that every set can be wellordered. Let $\mathsf{ZFC}^\bullet$ be the theory
\[ \mathsf{ZFC}^- + \ \pow(\pow(\omega_1)) \text{ exists } + \  \NS \text{ is precipitous.}\]

Whenever $\mathscr{H}$ is a transitive model of $\mathsf{ZFC}^-$ and $M \prec \mathscr{H}$ is  a countable elementary substructure, then $\delta(M)$ denotes the countable ordinal $M\cap \omega_1^\mathscr{H},$ and for every 
$x \in \bigcup M,$ we define
\[\Hull(M,x) = \{ f(x) : f \in M, x\in \dom(f) \}.\]
Note that under these conditions, $\Hull(M,x)$ is the minimal countable elementary submodel of $\mathscr{H}$ that $\subseteq$-extends $M$ and contains $x$ as an element.
The meaning of the statement	\[ x \parallel_{\omega_1} M, \]
is by definition that
\[\Hull(M,x) \cap \omega_1^\mathscr{H} = M \cap \omega_1^\mathscr{H}.\]

\begin{definition}\label{defselfgeneric}
A countable model $M\prec \mathscr{H} \vDash \mathsf{ZFC}^\bullet$ is called \emph{selfgeneric} if for every $\mathcal{A}\in M$ such that 
\[ \mathscr{H} \vDash \mathcal{A}\subseteq \NS^+  \text{ is a maximal antichain},\]
there exists $S\in \mathcal{A}\cap M$ such that $\delta(M)\in S$.  
\end{definition} 

\begin{definition}
	Let $\mathscr{M}$ be a countable transitive model of the theory $\mathsf{ZFC}^\bullet.$

	A \emph{generic iteration} of $\mathscr{M} = \mathscr{M}_0$ of length $\gamma$ (with $\gamma\leq \omega_1+1$) is a commuting system $\iterat = (\mathscr{M}_\beta, j_{\alpha\beta},G_\alpha : \alpha < \beta < \gamma)$ satisfying
	\begin{enumerate}[label=(\alph*)]
		\item for every $\alpha + 1 < \gamma$, $G_\alpha\subset (\NS^+)^{\mathscr{M}_\alpha}$ is a filter that is $(\NS^+)^{\mathscr{M}_\alpha}$-generic over $\mathscr{M}_\alpha,$
		\item for every $\alpha + 1 < \gamma$, $\mathscr{M}_{\alpha+1}$ is the (transitivised) generic ultrapower of $\mathscr{M}_\alpha$ induced by $G_\alpha$, $j_{\alpha, \alpha+1}$ is the corresponding ultrapower map,
		\item  for every limit ordinal $\beta <\gamma$, $(\mathscr{M}_\beta, (j_{\alpha\beta})_{\alpha < \beta} )$ is the (transitivised) direct limit of the system $(\mathscr{M}_{\beta'}, j_{\alpha\beta'} : \alpha < \beta' < \beta)$. 
	\end{enumerate}
	
	The $\mathsf{ZFC}^\bullet$-model $\mathscr{M}$ is called \emph{iterable} if for every limit ordinal $\gamma \leq \omega_1$ and every generic iteration  $\iterat = (\mathscr{M}_\beta, j_{\alpha\beta},G_\alpha : \alpha < \beta < \gamma)$  of length $\gamma$ of $\mathscr{M}$, the direct limit of the system  $(\mathscr{M}_\beta, j_{\alpha\beta}: \alpha < \beta < \gamma)$  is well-founded. 
\end{definition}

We will make use of the following standard way of constructing iterable models from sharps.
\begin{lemma}\label{lem trans_coll}
	
	Suppose $\mathscr{H}$ is a transitive $\mathsf{ZFC}^\bullet$-model with $\omega_1 \subseteq \mathscr{H}$ and $\theta \in \mathscr{H}$ is an ordinal such that
	\begin{enumerate}[label=(\alph*)]
		\item $\mathscr{H} \vDash \theta $ is a regular cardinal and $\theta > |\pow(\pow(\omega_1))|,$
		\item $\mathscr{H} \vDash$ $H_\theta$ and $H_\theta^\#$ both exist.
	\end{enumerate}
	If $M\prec \mathscr{H}$ is countable and contains $\theta,$ then the transitive collapse of $M \cap H^\mathscr{H}_\theta$ is an iterable $\mathsf{ZFC}^\bullet$-model.
\end{lemma}

A proof of Lemma~\ref{lem trans_coll} can be found by applying standard modifications to the argument given in \cite{woodin} to prove \cite[Theorem 3.14]{woodin}.   

\begin{definition}
	Let $\mathscr{H}$ be a transitive $\mathsf{ZFC}^\bullet$-model and let $(M_\alpha: \alpha < \omega_1)$ be a continuous $\subseteq$-increasing sequence of countable elementary submodels of $\mathscr{H}.$ We call $(M_\alpha: \alpha < \omega_1)$ an \emph{iteration-inducing filtration} of $\mathscr{H}$ if
	\[ \bigcup_{\alpha < \omega_1} M_\alpha = \mathscr{H}, \]
	and additionally the following two conditions hold for every $\alpha < \omega_1:$
	\begin{itemize}
		\item $M_\alpha$ is selfgeneric,
		\item $M_{\alpha+1}=\Hull(M_\alpha,\delta(M_\alpha)).$
	\end{itemize}
\end{definition}

There is a straightforward way to translate between generic iterations of length $\omega_1$ and iteration-inducing filtrations. Indeed, the following is easily checked.

\begin{lemma}\label{lem unique_generic_iteration}
	Let $\mathscr{H}$ be a transitive $\mathsf{ZFC}^\bullet$-model.
	\begin{enumerate}
		\item If $\iterat = (\mathscr{M}_\beta, j_{\alpha\beta},G_\alpha : \alpha < \beta < \omega_1)$ is a generic iteration of length $\omega_1$ with (transitivised) direct limit equal to $\mathscr{H},$ then $(j_{\alpha\omega_1}[\mathscr{M}_\alpha] : \alpha < \omega_1)$ is an iteration-inducing filtration of $\mathscr{H},$
		\item if $(M_\alpha: \alpha < \omega_1)$ is an iteration-inducing filtration of $\mathscr{H},$ then there exists a unique generic iteration $\iterat = (\mathscr{M}_\beta, j_{\alpha\beta},G_\alpha : \alpha < \beta < \omega_1)$ of length $\omega_1$ such that the following diagrams $(\alpha<\beta < \omega_1)$ commute:
		
		\[\begin{array}{lcl}
			M_\alpha &\quad \subseteq \quad& M_\beta\\
			\bigg\downarrow \pi_\alpha && \bigg\downarrow \pi_\beta\\
			\mathscr{M}_\alpha& \quad\xrightarrow{j_{\alpha\beta}}\quad &\mathscr{M}_\beta
		\end{array}\]
	with $\pi_\alpha:M_\alpha \to \mathscr{M}_\alpha$ for each $\alpha < \omega_1$ being the transitivising isomorphism. Moreover, the direct limit of the system $\iterat$ is equal to $\mathscr{H}.$
	\end{enumerate} 
\end{lemma}

\subsection{The second uniform indiscernible} If $x^\#$ exists for every real, then there exists a proper class of ordinals which are $x$-indiscernible for every real $x,$ these are called the uniform indiscernibles. Clearly, every uncountable cardinal is a uniform indiscernible and $\omega_1$ is the minimal uniform indiscernible ordinal. The second uniform indiscernible ${\bf u}_2 \leq \omega_2$ deserves special interest because it is simultaneously the second projective ordinal $\boldsymbol{\delta}_2^1$. These projective ordinals 
\[\boldsymbol{\delta}_n^1 = \sup\{ \alpha : \alpha \text{ is the length of a } \boldsymbol{\Delta}^1_n \text{ prewellordering of } \mathbb{R}  \} \]
were introduced by Moschovakis (\cite{moschovakis}) to gauge the extent to which $\mathsf{CH}$ can fail in a projectively definable way. One can check that if the sharp of every real exists, additionally
\[ {\bf u}_2 = \sup\{ \omega_1^{+ L[x]} : x \in \mathbb{R} \}.\]
Results of respectively Martin (\cite{martin}) and Woodin (\cite[Theorem 3.17]{woodin}), give that under both $(\mathsf{ZF}+)\mathsf{AD}$ and under $\NS$ saturated + $\pow(\omega_1)^\#,$ the second uniform indiscernible ${\bf u}_2$ has the  maximal value it could have, namely ${\bf u}_2 = \omega_2.$
 We need the following lemma which is also essentially used in Woodin's argument (\cite[Theorem 3.17]{woodin}) for\[\NS \text{ saturated } + \pow(\omega_1)^\# \text{ exists }\Rightarrow\  {\bf u}_2 = \omega_2. \]
\begin{lemma}[{\cite[Lemma 3.15]{woodin}}]
	\label{lem itbounded}
	Let $\mathscr{M}$ be a countable transitive model of $\mathsf{ZFC}^\bullet$ that is iterable. Let $x$ be a real that codes both $\mathscr{M}$ and some countable ordinal $\alpha.$ Then for every generic iteration $\iterat = (\mathscr{M}_\beta, j_{\alpha\beta},G_\alpha : \alpha < \beta < \alpha +1)$ of  $\mathscr{M}$ of length $\alpha + 1,$
	\[ \rank(\mathscr{M}_{\alpha}) < \omega_1^{+ L[x]}. \]
\end{lemma}

Finally, we will also make use of the well known fact that existence of sharps of all (transitive) sets in a fragment of the universe is not destroyed by small forcing. See e.g.\ \cite[Lemma 3.4]{Schlicht} for a much finer result of this type, for the present purposes the following lemma suffices.

\begin{lemma} \label{lem sharps in extension}
Let $\kappa$ be a cardinal and suppose that $H_{\kappa^+}^\#$ exists. Then for every forcing $\mathbb{P}$ of size at most $\kappa$ and every filter $G$ that is $\mathbb{P}$-generic over $V,$
\[V[G] \vDash X^\# \text{ exists for every transitive set } X \in H_{\kappa^+}.\]
\end{lemma}

\section{The strongly ssp property}\label{S.strong ssp}

\begin{notation}
	For a set $X$ and a cardinal $\theta,$ we write $X \ll \theta$ to mean that \[\pow(\pow(X)) \in H_\theta.\]
\end{notation}

\begin{definition}
	Given a $\bbP$-condition $p^*,$ the condition $p^*$ is called $(M, \mathbb{P})$-\emph{semigeneric} if for every $q\leq p^*$ and every maximal antichain $\mathcal{A} \subseteq \bbP$ with $\mathcal{A} \in M,$ there exists some $r \in \mathcal{A}$ such that
	\[ r \parallel q \text{ and } r \parallel_{\omega_1} M.\] Equivalently, $p^*$ is $(M, \mathbb{P})$-semigeneric iff for every $\mathbb{P}$-name $\tau \in M$ for a countable ordinal, 
	\[ p^\ast \Vdash \tau < \delta(M). \]
	The condition $p^*$ is called \emph{strongly $(M, \mathbb{P})$-semigeneric} if for every $q\leq p^*$ there exists $\trace(q | M) \in \bbP$ such that $\trace(q | M) \parallel_{\omega_1} M $ and for every $r \in \Hull(M,\trace(q | M)) \cap \bbP,$ 
	\[r \leq \trace (q | M)\Rightarrow r \parallel q. \]
	$\mathbb{P}$ is said to be \emph{(strongly) semiproper} with respect to $M\prec H_\theta$ if for every $p\in \mathbb{P}\cap M,$ there exists $p^* \leq p$ such that $p^*$ is (strongly) $(M, \mathbb{P})$-semigeneric.
\end{definition}

Recall that a set $X \subseteq [H_\theta]^\omega$ is \emph{projective stationary} if for every stationary subset $S \subseteq \omega_1$ the set $\{M \in X : M \cap \omega_1 \in S\}$ is stationary in $[H_\theta]^\omega.$

\begin{lemma}
	If $\mathbb{P}$ is strongly semiproper with respect to $M\prec H_\theta,$ then $\mathbb{P}$ is semiproper with respect to $M\prec H_\theta.$
	If $\mathbb{P}$ is semiproper with respect to projective stationary many $M \prec H_\theta,$ then  $\mathbb{P}$ is ssp.
\end{lemma}
\begin{proof}
	The first statement is clear because every condition that is strongly $(M, \mathbb{P})$-semi\-generic is also $(M, \mathbb{P})$-semigeneric. To prove the second statement, let be given a stationary set $S\subseteq \omega_1,$ a name for a club $\dot{C} \in H_\theta$ and an arbitrary $p\in \mathbb{P}.$ By assumption there exists $M\prec H_\theta$ such that $\bbP$ is semiproper with respect to $M$ and such that $p,\dot{C} \in M$ and $M\cap \omega_1 \in S.$ By semiproperness there exists $p^* \leq p $ that is $(M, \mathbb{P})$-semigeneric. Then $p^*$ forces that $M[G] \cap \omega_1 = M\cap \omega_1$ and hence also that $M\cap \omega_1 \in \dot{C}.$ 
\end{proof}

In fact (see \cite[Lemma 4.8]{fengjechzap}), a forcing is ssp if and only if for every sufficiently large regular cardinal $\theta,$ there exist projective stationary many countable $M \prec H_\theta$ such that $\bbP$ is semiproper with respect to $M.$ Replacing here semiproper by \emph{strongly} semiproper results in the following stronger notion.

\begin{definition}
	A forcing $\mathbb{P}$ is \emph{strongly ssp} if for every regular cardinal $\theta$ which is sufficiently large ($\theta \gg \bbP$), there exist projective stationary many countable $M \prec H_\theta$ such that $\mathbb{P}$ is strongly semiproper with respect to $M.$
\end{definition}

The following lemma makes apparent that even ccc forcings and countably closed forcings can fail to be strongly ssp.

\begin{lemma}
	If $\mathbb{P}$ is strongly ssp, then  $\mathbb{P}$ does not add any fresh functions $\omega_1 \to V.$
\end{lemma}
\begin{proof}
Let $\dot{f}$ be a $\bbP$-name and suppose $p \in \bbP$ such that
\[p \Vdash \dot{f} : \omega_1 \to V \quad \text{and} \quad (\forall \alpha < \omega_1)\, \dot{f} \upharpoonright \alpha \in V.\]
Choose a sufficiently large regular $\theta$ and let $M\prec H_\theta$ countable such that
$\mathbb{P}, \dot{f}, p \in M$ and such that $\mathbb{P}$ is strongly semiproper with respect to $M.$ Then pick $p^* \leq p$ which is strongly $(M, \bbP)$-semigeneric and also pick $q \leq p^*$ and $g : M \cap \omega_1 \to V$ such that
\[ q \Vdash \dot{f} \upharpoonright M \cap \omega_1 = g.  \]
 We claim that $\trace(q | M)$ decides $\dot{f}.$\\ Suppose not, then there exist $r_1, r_2 \in \bbP\, \cap\, \Hull(M, \trace(q | M)),$ $\alpha \in M \cap \omega_1$ and $x \neq y$ such that
 \[ r_1,r_2 \leq \trace(q | M), \qquad r_1 \Vdash \dot{f}(\alpha) = x, \qquad r_2 \Vdash \dot{f}(\alpha) = y. \]
 This contradicts $r_1 \parallel q$, $r_2 \parallel q$ and $q\Vdash \dot{f}(\alpha) = g(\alpha).$
\end{proof}

As a consequence, a strongly ssp forcing adds no new uncountable branches to a ground-model tree of height $\omega_1.$
\\

\begin{remark}
Of course, the strongly ssp property is inspired by Mitchell's (\cite{mitchell}) notion of strong properness. A forcing $\mathbb{P} \in H_\theta$ is strongly proper with respect to $M \prec H_\theta$ if for every $p \in M \cap \mathbb{P},$ there exists $p^* \leq p$ that is strongly $(M, \mathbb{P})$-generic.   $p^*$ being strongly $(M, \mathbb{P})$-generic means that for every $q\leq p^*,$ there exists $\trace (q | M) \in \mathbb{P} \cap M$ such that
\[(\forall r \in \mathbb{P} \cap M)\ r \leq \trace(q | M) \Rightarrow r \parallel q.\]
$\mathbb{P}$ is strongly proper if for every regular $\theta \gg \bbP$ there exist club many countable $M \prec H_\theta$ such that $\mathbb{P}$ is strongly proper with respect to $M.$
$\,$\\
Unsurprisingly, it is easy to check that every strongly proper forcing is strongly ssp. 
\end{remark}

\section{Games: catching and capturing}
In this section three types of games will be introduced,
\begin{itemize}
	\item the antichain catching games $\mathbf{G}^{cat}_M,$
	\item the capturing games $\mathbf{G}^{cap}_M(X),$
	\item and the catching-capturing games
	$\mathbf{G}^{c\text{-}c}_M(X).$
\end{itemize}

All three types of games will be played by two players $P1$ (\emph{the 
		Challenger}) and $P2$ (\emph{the Constructor}). $P1$ and $P2$ thank their nicknames to the fact that the moves of $P2$ are used to construct a growing sequence $(M_n:n<\omega)$ of countable elementary submodels of some $H_\theta,$ while the moves of $P1$ can be thought of as challenges for $P2$ to make the models $M_n$ satisfy certain conditions. All games considered have length $\omega.$ If $\mathbf{G}$ is one of these games, we write $P1 \nearrow \mathbf{G}$ for the statement that $P1$ has a winning strategy for the game $\mathbf{G}$ (and likewise for $P2 \nearrow \mathbf{G}$).  
	
	We proceed straight to the definitions. 	
\begin{definition}[The antichain catching game]
	Let $\theta$ be a regular cardinal and suppose $\pow(\omega_1) \in H_\theta.$ Let also $M \prec H_\theta$ be countable. The \emph{antichain catching game} $\mathbf{G}^{cat}_M$ for $M$ is defined as follows.	
Define $M_0 = M.$ In round $n$ of the game, \begin{itemize}
		\item $P1$ plays some maximal antichain $\mathcal{A}_n\subseteq \NS^+$ 
		with $\mathcal{A}_n\in M_n$,
		\item $P2$ has to play some $ S_n\in\mathcal{A}_n$ such that $M_n \cap \omega_1 \in S_n$ and $S_n \parallel_{\omega_1} M_n$,
		\item define $ M_{n+1}=\Hull(M_n,S_n)$.
	\end{itemize}
Thus, a play of the game $\mathbf{G}^{cat}_M$ can be pictured as follows.
	\medskip
	\begin{center} 
		$\begin{array}{lcccccc}
			\phantom{P1} & M_0 &\stackrel{+S_0}{\rightsquigarrow}& M_1 &\stackrel{+S_1}{\rightsquigarrow}&M_{2} &{\ldots}\\
			
			\phantom{P2} & \vin&\phantom{S_0}&\vin&\phantom{f_1}&\vin&\phantom{\ldots}
		\end{array}$\\
		
		$ \begin{array}{l|cccccc}
			P1 & \mathcal{A}_0 &\phantom{\stackrel{+S_0}{\rightsquigarrow}}& \phantom{i}\mathcal{A}_1 &\phantom{\stackrel{+f_1}{\rightsquigarrow}}&\phantom{i}\mathcal{A}_{2} &\ldots\\
			\hline\\ [-0.7em]
			P2 & &S_0&\phantom{nn}&S_{1}&&\ldots
		\end{array}$
	\end{center}
\smallskip

\noindent
		A winner is assigned as decided by the following rules.
		\begin{itemize}
			\item[] If there is a player who first disobeys the rules, this player loses,

\item[] else, 
$P2$ wins.
\end{itemize}
\end{definition}

The catching game is clearly just one of the many variations on the classic theme of antichain catching and it is easy to observe that $P1$ has a winning strategy for the catching game $\mathbf{G}^{cat}_M$ if and only if there exists some countable selfgeneric $N \prec H_\theta$ such that $M \subseteq N$ and $M\cap \omega_1 = N\cap \omega_1.$ 

\begin{definition}[The capturing game]
Let $\theta$ be a regular cardinal and $X \in H_\theta$. Suppose $M \prec H_\theta$ is countable and $X\in M$. The $X$-\emph{capturing game} $\mathbf{G}^{cap}_M(X)$ for $M$ is defined as follows.
Define $M_0=M$.	
In round $n$ of the game, \begin{itemize}
	\item $P1$ plays some $x_n\in X$ ($x_n$ not necessarily belonging to $M_n$),
	\item $P2$ then has to play some function $f_n:\omega_1 \to X$, with $f_n(M_n\cap\omega_1)=x_n$ and $f_n \parallel_{\omega_1} M_n$,
	\item define $ M_{n+1}=\Hull(M_n,f_n)$.
\end{itemize}
Thus, a play of the game $\mathbf{G}^{cap}_M(X)$ can be pictured as follows.
\medskip
\begin{center} 
	$\begin{array}{lcccccc}
		\phantom{P1} & M_0 &\stackrel{+f_0}{\rightsquigarrow}& M_1 &\stackrel{+f_1}{\rightsquigarrow}&M_{2} &{\ldots}\\
		
	\end{array}$\\
	
	$ \begin{array}{l|cccccc}
		P1 & x_0 &\phantom{\stackrel{+f_0}{\rightsquigarrow}}& \phantom{i}x_1 &\phantom{\stackrel{+f_1}{\rightsquigarrow}}&\phantom{i}x_{2} &\ldots\\
		\hline\\ [-0.7em]
		P2 & &f_0&\phantom{nn}&f_{1}&&\ldots
	\end{array}$
\end{center}
\smallskip

\noindent
The winner is again assigned in accordance with the following same rules.
\begin{itemize}
	\item[] If there is a player who first disobeys the rules, this player loses,
	\item[] else, 
	$P2$ wins.
\end{itemize}
\end{definition}

Finally, the catching-capturing game $\mathbf{G}^{c\text{-}c}_M(X)$ is simply a mix of the two above games $\mathbf{G}^{cat}_M$ and $\mathbf{G}^{cap}_M(X)$. 

\begin{definition}[The catching - capturing game]
	Let $\theta$ be a regular cardinal with\linebreak $\pow(\omega_1) \in H_\theta$. Suppose $M \prec H_\theta$ is countable and $X\in M$.
	The \emph{catching-capturing game} $\mathbf{G}^{c\text{-}c}_M(X)$ for $M$ is defined as follows.
	In round $n,$ player $P1$ now has two options:
	\begin{itemize}
		\item either $P1$ makes a move $ \mathcal{A}_n \in M_n$ as in the game  $\mathbf{G}^{cat}_M$, 
		\begin{itemize}			
			\item the player $P2$ then has to answer with some $S_n,$ respecting the same restrictions as in the game $\mathbf{G}^{cat}_M$,
			\item define $ M_{n+1}=\Hull(M_n,S_n)$.
		\end{itemize}
		\item or $P1$ decides to play some $x_n \in X$, as in the game $\mathbf{G}^{cap}_M(X)$.
		\begin{itemize}
			\item $P2$ then has to answer with some function $f_n: \omega_1 \to X$, respecting the same restrictions as in the game $\mathbf{G}^{cap}_M(X)$,
			\item define $ M_{n+1}=\Hull(M_n,f_n)$.
		\end{itemize}
	\end{itemize}
	
	\medskip
	
	A play of the game $\mathbf{G}^{c\text{-}c}_M(X)$ can be pictured as follows.
	
	\begin{center}
		
		$\begin{array}{lcccccccc}
			\phantom{P1} & M_0 &\stackrel{+S_0}{\rightsquigarrow}& M_1 &\stackrel{+f_1}{\rightsquigarrow}&M_{2}&\stackrel{+S_2}{\rightsquigarrow}& M_{3} &\phantom{\ldots}\\
			
			\phantom{P2} & \vin&\phantom{S_0}&&\phantom{f_1}&\vin&\phantom{S_{2}}&&\phantom{\ldots}
		\end{array}$\\
		
		$\begin{array}{l|cccccccc}
			P1 & \mathcal{A}_0 &\phantom{\stackrel{+S_0}{\rightsquigarrow}}& \phantom{n}x_1 &\phantom{\stackrel{+f_1}{\rightsquigarrow}}&\mathcal{A}_{2}&\phantom{\stackrel{+S_0}{\rightsquigarrow}}& \phantom{n}x_{3} &\ldots\\
			\hline\\ [-0.7em]
			P2 & &S_0&&f_1&\phantom{nn}&S_{2}&&\ldots
		\end{array}$
		
	\end{center}

The winning conditions remain unaltered:
\begin{itemize}
	\item[] If there is a player who first disobeys the rules, this player loses,
	\item[] else,
	$P2$ wins.
	\end{itemize} 
	
\end{definition}

Note that the games $\mathbf{G}^{cat}_M, \mathbf{G}^{cap}_M(X),\mathbf{G}^{c\text{-}c}_M(X)$ are all determined, the winning conditions for $P2$ being closed. We will allow countable models $M \prec H_\theta$ to occur as side conditions in the forcing $\theforcingg$ only if $P2 \nearrow \mathbf{G}^{c\text{-}c}_M(X)$ for a sufficiently large set $X \in H_\theta.$ Thus it will be crucial for this approach that sufficiently many models exist for which $P2$ has a winning strategy in $\mathbf{G}^{c\text{-}c}_M(X)$. If the non-stationary ideal on $\omega_1$ is precipitous, this is warranted by the following proposition.

\begin{proposition}\label{prop projstatmany}
Suppose $\theta \gg \omega_1$ is a regular cardinal and $X \in H_\theta$. The following are all equivalent.
\vspace{0.5em}
		\begin{enumerate}[label=(\alph*)]
		 \setlength\itemsep{0.5em}
			\item there exist projective stationary many countable $M\prec H_\theta$ for which $P2 \nearrow \mathbf{G}^{c\text{-}c}_M(X),$
			\item there exist projective stationary many countable $M\prec H_\theta$ for which $P2 \nearrow \mathbf{G}^{cat}_M$,
			\item there exist projective stationary many countable selfgeneric $M\prec H_\theta,$
			\item $\NS$ is precipitous.
\end{enumerate}
\end{proposition}

\begin{proof}$\;$

\noindent		
$	\boxed{(d)\Rightarrow (a)}$\vspace{0.05cm}
We first prove that $(d)$ implies $(a)$. The argument is a  relative of those in \cite[Theorem 3.8]{cox_zeman} and \cite[Proposition 2.4]{ketlarzap}.\\	
	Let $A \subseteq \omega_1$ be stationary and let $F: H_\theta \to H_\theta$ be arbitrary. We use an elementarity-argument to prove that there exists a countable model
	$M \prec H_\theta$ which is closed under $F$ and satisfies $M \cap \omega_1 \in A$ and $P2 \nearrow \mathbf{G}^{c\text{-}c}_M(X).$ 
	Let $G \subseteq \NS^+$ be a filter generic over $V$ and containing $A$. Using precipitousness of $\NS,$ let $W$ be the corresponding transitivised ultrapower of $V$ and let $j:V\to W$ be the corresponding elementary embedding.
	Select (in $V$) an arbitrary countable elementary submodel $N \prec H_\theta$ that contains $X$ and has the property that for every function $f: \omega_1 \to H_\theta$ which is in $N,$ also the function\linebreak $F\circ f: \omega_1 \to H_\theta$ is in $N.$\\
	Next, we continue working in $W.$ Note that, in $W$, $j(N)$ is a countable elementary submodel of $H_{j(\theta)}$. Hence, the model
	\[\widetilde{N} := \{ f(\xi) : \xi \in \omega_1^V, f \in j(N) \text{ with } \xi \in \dom(f)\},\]
	is, in $W$, a countable elementary submodel of $H_{j(\theta)}.$\\ By construction we have that $\widetilde{N} \cap \omega_1^W = \omega_1^V \in j(A)$ and it clearly follows from the choice of $N$ that $\widetilde{N}$ is closed under the function $j(F):H_{j(\theta)} \to H_{j(\theta)}.$  Using elementarity of the embedding $j,$ the proof will be completed if we can establish that
	\[(\text{ the model } \widetilde{N} \prec H_{j(\theta)} \text{ satisfies } P2 \nearrow\mathbf{G}^{c\text{-}c}_{\widetilde{N}}(j(X))\, )^W.\]
	Arguing by contradiction, suppose that this is not the case. Then there exists in $W$ some winning strategy $\Sigma$ for $P1$ (the Challenger) in the game $\mathbf{G}^{c\text{-}c}_{\widetilde{N}}(j(X)).$ Let $T$ be the tree (as defined in $W$) of all partial plays of the game $\mathbf{G}^{c\text{-}c}_{\widetilde{N}}(j(X))$ in which $P1$ plays by the strategy $\Sigma$ and in which $P2$ has not lost yet, then
	\[(\text{ the tree } T \text{ is well-founded })^W.\]
	We now leave $W$ and describe in $V[G]$ an infinite branch $B$ in the tree $T,$ thus arriving at the desired contradiction (since $W$ is transitive, well-foundedness is absolute between $W$ and $V[G]$). The branch $B$ can be defined as follows:
	
	\begin{center}
		\begin{minipage}{0.90\textwidth}
			\begin{itemize}
	\item		$P1$ makes a $\mathbf{G}^{cap}_{\widetilde{N}}(j(X))$-move in round $n$:\\
			Suppose $P1$ plays the set $x \in j(X).$ There exists $f: \omega_1^V \to X$ such that $x= j(f)(\omega_1^V).$\\ $P2$ replies by playing the function $j(f).$
			\end{itemize}
		\end{minipage}
	\end{center}
	
	\begin{center}
		\begin{minipage}{0.90\textwidth}
			\begin{itemize}
				\item
			$P1$ makes a $\mathbf{G}^{cat}_{\widetilde{N}}$-move in round $n$:\\
			Suppose $P1$ plays the maximal antichain $\mathcal{A}_n,$ then it is of the form $j(f)(\xi,j(\bar{y})),$ where $j(\bar{y})$ is the sequence of moves that have been played by $P2$ in the first $n$ moves, $\xi \in \omega_1^V$, and where 
			\[f: \dom(f) \to \{ \mathcal{A} \subseteq \NS^+ : \mathcal{A }\text{ is a maximal antichain in } \NS^+ \}^V\] is a function which belongs to $N.$ Then \[\mathcal{B} := f(\xi,\bar{y}) \in \{ \mathcal{A} \subseteq \NS^+ : \mathcal{A }\text{ is a maximal antichain in } \NS^+ \}^V,\] so there exists some $S\in \mathcal{B} \cap G$.\\ $P2$ replies by playing the set $j(S).$
			\end{itemize}
		\end{minipage}
	\end{center}
	
	\noindent
	To check that $P2$ does not lose by playing this way, we note that
	\[ M_n \subseteq \{ f(\xi, j(w)): \xi \in \omega_1, w \in H_\theta^{< \omega}, f \in j(N)= j[N] \} \subseteq j [H_\theta]\]
	and
	\[j [H_\theta] \cap \omega_1^W = \omega_ 1^V = \widetilde{N} \cap \omega_1^W.\]

\noindent		
$	\boxed{(a)\Rightarrow (b)}$
The implication $(a)\Rightarrow(b)$ is trivial.	

\noindent		
$	\boxed{(b)\Rightarrow (c)}$
It was already noted above that 
\[ P2 \nearrow \mathbf{G}^{cat}_M\] is equivalent to the statement that there exists an extension of $M$ to a countable selfgeneric $N\prec H_\theta$ containing the same countable ordinals. It is easily seen to follow from this that $(b)$ and $(c)$ are equivalent.
	
	\noindent		
	$	\boxed{(c)\Rightarrow (d)}$\vspace{0.05cm}
Finally, the direction $(c)\Rightarrow(d)$ is well known and can be found in \cite[Lemma~3.6]{cox_zeman}.  
We repeat this short argument for the sake of completeness.
Suppose $\NS$ were not precipitous. There would exist $S \in  \NS^+$ such that 
\begin{equation*}
\begin{split}
	S \Vdash \text{there exists an infinite }<_G&\text{-decreasing} \text{ sequence}\\ &\text{ of ground-model functions } f_n: \omega_1 \to  \Ord .
	\end{split}
\end{equation*}
 Let $M\prec H_\theta$ be countable and selfgeneric with $S \in M$ and $\delta := M\cap \omega_1 \in S.$
Let $\pi: M \to \mathscr{M}$ be the transitive collapse of $M.$ We can assume $\theta$ was sufficiently large so that inside the transitive $\mathsf{ZFC}^\bullet$-model $\mathscr{M}$, $\pi(S)$ is a stationary subset of $\omega_1$ and
\begin{equation*}
	\begin{split}
\pi(S) \Vdash \text{there exists an infinite } <_G&\text{-decreasing} \text{ sequence}\\ &\text{  of ground-model functions } f_n: \omega_1 \to  \Ord .
	\end{split}
\end{equation*}
Moreover, by selfgenericity,
\[ G_\delta = \{ \pi(T) : T\in \NS^+ \cap M \text{ with } \delta \in T\}\]
is a $\pi(\NS^+)$-filter that is generic over $\mathscr{M}$. This contradicts the prior statement about $\pi(S)$ forcing ill-foundedness of $\Ult(V,G)$ because $\pi(S) \in G_\delta$ and $\Ult(\mathscr{M}, G_\delta)$ is well-founded as it is $\in$-isomorphic to $\Hull(M,\delta).$

\end{proof}


Although not needed for the discussion in the next section of the forcing $\theforcingg,$ we end this section with some results that show that the equivalences in Proposition~\ref{prop projstatmany}
fall apart when one replaces \emph{projective stationary many} by \emph{club many}.

It is well known that the existence of club many selfgeneric models is equivalent to saturation of the non-stationary ideal (see e.g.\ \cite[Vol.\ 2, Lemma 13.3.46]{foreman}). Therefore, if $\NS$ is saturated there exist club many countable $M\prec H_\theta$ for which $P2 \nearrow \mathbf{G}^{cat}_M.$
In contrast, the existence of club many countable models $M \prec H_\theta$ for which $P2 \nearrow \mathbf{G}^{cap}_M(X)$ can be translated into strong versions of Chang's conjecture.

\begin{definition}($\textrm{cof}$-Strong Chang's Conjecture, \cite[Definition 4.3]{cox})
 $\mathsf{SCC}^{\textrm{cof}}$ is the statement that for all regular cardinals 
$\theta \gg \omega_2$ and every countable $M \prec H_\theta$ there exist cofinally many $\xi \in \omega_2$ such that $\xi  \parallel_{\omega_1} M.$
\end{definition}

\begin{definition}(Global Strong Chang's Conjecture, \cite[Definition 5.6]{doebler_schindler}, \cite[Definition~4.9]{cox})
	$\textrm{Glob-}\mathsf{SCC}^{\textrm{cof}}$ is the statement that for all regular cardinals $\theta \gg \lambda \geq \omega_2,$ for every\linebreak $\lambda \in M \prec H_\theta$ and every $X \in [H_\lambda]^{\omega_1},$ there exists $Y \in [H_\lambda]^{\omega_1}$  such that
	\[ X \subseteq Y \text{ and } Y \parallel_{\omega_1} M.\]
\end{definition}

\begin{lemma} \label{lem capturing countable ordinals}
	The following two statements are equivalent.
	\begin{enumerate}[label=(\alph*)]
		\item $\mathsf{SCC}^{\textrm{cof}},$
		\item for every regular $\theta \gg \omega_2$ and every $M \prec H_\theta$ countable, $P2$ has as a winning strategy in the $\omega_2$-capturing game $\mathbf{G}^{cap}_M(\omega_2)$ for $M$.
		
	\end{enumerate}
\end{lemma}

\begin{proof}$\,$\\
	$\boxed{(a) \Rightarrow (b)}$ Suppose first $P1$ plays $\rho \in \omega_1.$ 
	Use $\mathsf{SCC}^{\textrm{cof}}$ to find $M \subseteq M^* \prec H_\theta$ such that\linebreak $M\cap\omega_1=M^*\cap\omega_1$ and such that there exists $\xi^* \in  M^* \cap \omega_2$ satisfying\linebreak 
	$\otp( M^* \cap \xi^* ) = \rho.$ Choose inside $M^*$ a bijection $g$ from $\omega_1$ onto $\xi^*.$  Let $f: \omega_1 \to \omega_1$ be defined by 
	$f(\alpha) = \otp(g[\alpha]).$
	Then $f \parallel_{\omega_1} M$ and $f(\delta(M)) = \rho.$\\ Suppose then that $P1$ plays $\rho \in \omega_2 \setminus \omega_1.$ By $\mathsf{SCC}^{\textrm{cof}}$, we can assume that there exists $\xi \in M$ such that $\rho < \xi.$ Choose in $M$ a bijection $g$ from $\omega_1$ onto $\xi.$ Using the argument in the foregoing case, one finds $f: \omega_1 \to \omega_1$ such that $f \parallel_{\omega_1} M$ and $f(\delta(M)) = g^{-1}(\rho).$\\
	$\boxed{(b) \Rightarrow (a)}$\vspace{0.05cm} Given $\rho \in \omega_2,$ it follows from $P2 \nearrow \mathbf{G}^{cap}_M(\omega_2)$ that there exists $f: \omega_1 \to \omega_2$ such that $f \parallel_{\omega_1} M$ and  $f(\delta(M)) = \rho.$  Then $\xi = \sup(f)$ satisfies $\xi > \rho$ and $\xi \parallel_{\omega_1} M.$			
\end{proof}

\begin{proposition}\label{prop globSCC games}
	The following two statements are equivalent.
	\begin{enumerate}[label=(\alph*)]
		\item $\textrm{Glob-}\mathsf{SCC}^{\textrm{cof}},$
		\item for all regular $\theta \gg \lambda \geq \omega_2,$ and every $M \prec H_\theta$ with $\lambda \in M,$ $P2$ has a winning strategy in the $H_\lambda$-capturing game $\mathbf{G}^{cap}_M(H_\lambda)$ for $M.$
	\end{enumerate}
\end{proposition}
\begin{proof}$\;$
	
	\noindent
	$\boxed{(a) \Rightarrow (b)}$ $\textrm{Glob-}\mathsf{SCC}^{\textrm{cof}}$ implies that for every $x \in H_\lambda$ there exists some $\rho < \omega_1$ and $f: \omega_1 \to  H_\lambda$ such that $f(\rho) = x$ and $f \parallel_{\omega_1} M.$ $\textrm{Glob-}\mathsf{SCC}^{\textrm{cof}}$ also implies $\mathsf{SCC}^{\textrm{cof}}$, which by Lemma~\ref{lem capturing countable ordinals} implies that there exists a function $g : \omega_1 \to \omega_1$ such that $g(\delta(M)) = \rho$ and $ \{ g,f \} \parallel_{\omega_1} M.$\\
	$\boxed{(b) \Rightarrow (a)}$\vspace{0.05cm} Given  $\lambda \in M \prec H_\theta$ and $X \in [H_\lambda]^{\omega_1},$ find a function $g: \omega_1 \to [H_\lambda]^{\omega_1}$ such that $g(\delta(M)) = X$ and $g \parallel_{\omega_1} M.$ Then $Y = \bigcup_{\alpha \in \omega_1} g(\alpha)$ satisfies
	\[|Y| = \aleph_1, \quad X \subseteq Y, \quad Y \parallel_{\omega_1} M.\]
\end{proof}

\section{The forcing $\theforcing{\lambda}{\mu}$}

	Throughout this section, fix two regular cardinals
	$\omega_1 \ll \lambda \ll \mu.$
We further write $\mathbf{G}^{c\text{-}c}_M$ for $\mathbf{G}^{c\text{-}c}_M(H_\lambda).$ If $z$ is a finite run of the game $\mathbf{G}^{c\text{-}c}_M,$ we write $P2(z)$ for the finite set of all moves that have been made by $P2$ in $z.$  The conditions in the forcing $\theforcing{\lambda}{\mu}$ will be finite sets of special triples,  we call those triples \emph{building blocks}. 
\begin{definition}
	A set $ B\in H_\mu$ is called a \emph{building block} if it is a three-tuple $ B=(M,z,\Sigma)$, with:
	\begin{enumerate}[label=(\alph*)]
		\item $ M\prec H_\mu$ countable with $\lambda \in M$ and such that $P2 \nearrow \mathbf{G}^{c\text{-}c}_M,$
		\item $ \Sigma$ a winning strategy for $P2$ in $ \mathbf{G}^{c\text{-}c}_M$,
		\item $ z$ a finite play of the game $ \mathbf{G}^{c\text{-}c}_M$ in which $P2$ follows $\Sigma$.		
	\end{enumerate}
\end{definition}

\begin{notation}
	Let $B = (M,z,\Sigma)$ and $B' = (M',z',\Sigma')$ be building blocks. We define
	\begin{itemize}
		\item  $\text{Ext}(B):=\Hull(M,P2(z))\prec H_\mu,$ 
		\item $B \sqsubseteq_{\text{game}} B'$ holds if and only if $M = M',\Sigma = \Sigma'$ and $z \subseteq z'.$
	\end{itemize}
\end{notation}

We are now ready to define the main forcing $\theforcing{\lambda}{\mu}$.

	\begin{definition}
	$p$ is a $\theforcing{\lambda}{\mu}$-\emph{condition} if
	\begin{enumerate}[label=(\alph*)]
		\item $p$ is a finite set of building blocks,
		\item for every two building blocks $B_1, B_2\in p$,
		\[ { \text{Ext}(B_1) \ni B_2 \text{ or } B_1=B_2 \text{ or } B_1 \in \text{Ext}(B_2)}. \]
	\end{enumerate}	
	The order on $\theforcing{\lambda}{\mu}$ is defined as follows. If $ p,q\in \theforcing{\lambda}{\mu}$, then 	
	\[{q \leq p\iff (\forall B\in p)(\exists B'\in q)B\sqsubseteq_{\text{game}} B'}. \]
\end{definition}

In analysing the forcing $\theforcing{\lambda}{\mu}$ it will be convenient to use the following additional notation.
\begin{notation}$\,$
\begin{itemize}
	\item Call a building block $B$ an $M$-block if and only if  $B = (M,z,\Sigma)$ for certain $z,\Sigma,$
	\item let $\kl$ be the strict partial order defined on  building blocks by
	 \[ B_1 \kl B_2 \iff B_1 \in \Ext(B_2). \] 
\end{itemize}
\end{notation}

\begin{lemma}\label{lem strongly_ssp}
Suppose $B = (M,z,\Sigma)$ is a building block and $p^*$ is a $\theforcing{\lambda}{\mu}$-condition that contains $B$ as an element. If $\mu \ll \theta$ regular and $N \prec H_\theta$ is countable with $\lambda,\mu \in N$ and satisfies 
\[ N \cap H_\mu = M,\]
then $p^*$ is strongly $(N,\theforcing{\lambda}{\mu})$-semigeneric. 
\end{lemma}

\begin{proof}
	To check that $p^*$ is strongly $(N, \theforcing{\lambda}{\mu})$-semigeneric, consider a $\theforcing{\lambda}{\mu}$-con\-dition $q \leq p^*$ and define 
	\[ \trace(q|N) = q \cap \Ext(B_M^q) \in \theforcing{\lambda}{\mu},\]
	where $B_M^q$ is the unique $M$-building block that belongs to $q.$ 
	From $\trace(q|N) \in \Ext(B^q_M)$ follows
	\[ \trace(q|N)  \parallel_{\omega_1} N .\]
	Suppose then that 
	$r \in \theforcing{\lambda}{\mu} \cap \Hull(N,\trace(q|N) )$ and $r \leq \trace(q|N).$ To prove that $r \parallel q,$ one checks that $r \cup (q \setminus \trace(q|N) )$ is a $\theforcing{\lambda}{\mu}$-condition. This follows because for every two building blocks $B_1 \in r$ and $B_2 \in q \setminus \trace(q|N),$ either
	$B_1 \kl B_M^q = B_2$ or $B_1 \kl B_M^q \kl B_2.$
\end{proof}

The foregoing lemma allows to follow the well known side condition pattern for proving (semi)properness: \emph{to find $p^* \leq p$ that is a (semi)generic condition for $M,$ simply add $M$ on top of $p$.}

\begin{proposition}
If $\NS$ is precipitous, then $\theforcing{\lambda}{\mu}$ is strongly ssp.
\end{proposition}
\begin{proof}
Let $\theta \gg \mu$ be regular. Using Proposition \ref{prop projstatmany}, it suffices to check that $\theforcing{\lambda}{\mu}$ is strongly ssp with respect to every countable model $M \prec H_\theta$ that contains $ \lambda,\mu$ and satisfies $P2 \nearrow \mathbf{G}^{c\text{-}c}_M.$ Note that in this case also $P2 \nearrow \mathbf{G}^{c\text{-}c}_{M\cap H_\mu}.$ Thus, if $p \in M \cap \theforcing{\lambda}{\mu},$ choose $\Sigma_M$ such that $B_M = (M\cap H_\mu, \emptyset, \Sigma_M)$ is a building block, define $p^* =p \cup \{ B_M\}$ and note that $p^* \in \theforcing{\lambda}{\mu}.$ By Lemma \ref{lem strongly_ssp}, $p^*$ is strongly $(M,\theforcing{\lambda}{\mu})$-semigeneric.
\end{proof}

Throughout the rest of this section, we assume that $\NS$ is precipitous and thus 
that $\theforcing{\lambda}{\mu}$ is strongly ssp.
Under this assumption, we analyse the generic object.

\begin{definition}
Let $G$ be a $\theforcing{\lambda}{\mu}$-filter generic over $V.$
\begin{itemize}
	\item Define \[C^G=\{ \delta(M): (\exists p\in G) (\exists B\in p)\  B = (M,z,\Sigma) \text{ for certain }  z,\Sigma \}.\]	
\end{itemize}
\qquad	Clearly, $C^G$ is unbounded in $\omega_1,$ so
	\begin{itemize}
		\item let $(\delta_\alpha^G:{\alpha < \omega_1})$ be the increasing enumeration of $C^G$. 
	\item Define for every $\alpha <\omega_1$,
	\[ M_\alpha^G = \bigcup_{p\in G}\{  \Ext(B) \cap H_\lambda^V : B\in p \text{ is an }M\text{-block and } M\cap \omega_1 = \delta_\alpha^G \}. \]
\end{itemize}
\end{definition}

Note that inside $V[G],$ $M_\alpha^G$ is a countable elementary submodel of $H^V_\lambda$. Countability of $M_\alpha^G$ follows because for every two $p,q \in G,$ and every $(M_1,z_1,\Sigma_1) = B_1 \in p$,\linebreak $(M_2,z_2,\Sigma_2)  = B_2 \in q$:
\[ M_1\cap \omega_1 = M_2 \cap \omega_1 \quad \Rightarrow \quad B_1 \sqsubseteq_{\text{game}} B_2 \text{ or }  B_2 \sqsubseteq_{\text{game}} B_1. \] 

\begin{proposition}\label{prop it-ind fact}
	The sequence $(M_\alpha^G : \alpha < \omega_1)$ is an iteration-inducing filtration of~ $H_\lambda^V.$
\end{proposition}
\begin{proof}
	This is checked through the following series of density arguments.
	\begin{claim} $(M_\alpha^G:\alpha <\omega_1)$ is a continuous sequence of countable elementary substructures of $H_\lambda^V$ that together cover $H_\lambda^V$.
		\end{claim}
	\begin{proof}\renewcommand{\qedsymbol}{$\blacksquare$}
	First note that for every $x \in H_\lambda^V$ there are densely many $p \in \theforcing{\lambda}{\mu}$ that contain some block $B = (M,z,\Sigma)$ with $x \in M$.\\	Suppose next that $\beta <\omega_1$ is limit, then \[\bigcup_{\alpha <\beta}M^G_\alpha=M^G_\beta\]
		because\\
	$	\boxed{\subseteq}$ if $\alpha <\beta$ and $x\in M^G_\alpha$, then there exists $p\in G$ with $(M_1,z_1,\Sigma_1),(M_2,z_2,\Sigma_2)\in p$ such that $M_1\cap\omega_1=\delta_\alpha^G$, $M_2\cap\omega_1=\delta_\beta^G$ and
		$x\in\Hull(M_1,P2(z_1))\cap H_\lambda^V.$\\ Then $(M_1,z_1,\Sigma_1) \kl (M_2,z_2,\Sigma_2),$ so $x\in\Hull(M_2,P2(z_2))\cap H_\lambda^V$.\\
		$\boxed{\supseteq}$ if $x\in M^G_\beta$, then there exists $p\in G$ and $B_1,B_2\in p$ such that
		\begin{align*}
			B_1&=\max\{  B\in p: B \kl B_2 \},\\
			B_2&=(M_2,z_2,\Sigma_2),p \Vdash M_2\cap\omega_1=\delta_\beta^G,\ x\in\Ext(B_2),\\
			B_1&=(M_1,z_1,\Sigma_1).
		\end{align*}
		Consider the play $z_1'$ which is an extension of $z_1$ with one round in which $P1$ plays $x$ as a $\mathbf{G}^{cap}$-move and $P2$ answers following $\Sigma_1$.
		Then $p'=p\setminus \{ B_1\}\cup \{  (M_1,z_1',\Sigma_1) \}\in\theforcing{\lambda}{\mu}$ forces that $x\in \bigcup_{\alpha <\beta} M^G_{\alpha}$.
	\end{proof}
	
	\noindent
	\begin{claim} Every $M^G_\alpha$ is selfgeneric.\end{claim}
	\begin{proof}\renewcommand{\qedsymbol}{$\blacksquare$}
		Suppose that $p\in\theforcing{\lambda}{\mu}$ with $(M,z,\Sigma)\in p$, $p \Vdash M\cap\omega_1=\delta_\alpha^G$ and $\mathcal{A}$ is a maximal antichain of $\NS^+$ such that $\mathcal{A}\in\Hull(M,P2(z))$. Then consider the play $z'$ which is an extension of $z$ with one round in which $P1$ plays $\mathcal{A}$ as a $\mathbf{G}^{cat}$-move and $P2$ answers following $\Sigma$, say the answer is $S\in \mathcal{A}$. Then $p'=p\setminus \{ (M,z,\Sigma)  \}\cup\{ (M,z',\Sigma)  \} \in \theforcing{\lambda}{\mu}$ forces that $S\in M^G_\alpha\cap \mathcal{A}$ and $M^G_\alpha\cap \omega_1=\delta_\alpha^G\in S$.
	\end{proof}
	
	\noindent
	\begin{claim} $M^G_{\alpha+1}=\Hull(M^G_\alpha,\delta_\alpha^G)$.\end{claim}
	\begin{proof}\renewcommand{\qedsymbol}{$\blacksquare$}
		$\boxed{\supseteq}$ is already clear.\\
		$\boxed{\subseteq}$\vspace{0.05cm} Suppose that $p\in\theforcing{\lambda}{\mu}$ with $(M_1,z_1,\Sigma_1),(M_2,z_2,\Sigma_2)\in p$, $p \Vdash M_1\cap \omega_1=\delta_\alpha^G$,\linebreak $p \Vdash M_2\cap \omega_1=\delta_{\alpha+1}^G$ and $x\in\Hull(M_2,P2(z_2))\cap H_\lambda^V$ then consider the play $\tilde{z}_1$ which is an extension of $z_1$ with one round in which $P1$ plays $x$ as a $\mathbf{G}^{cap}$-move and $P2$ answers following $\Sigma_1$. Say this answer of $P2$ is the function $f$.
		Then $\tilde{p}=p\setminus \{  (M_1,z_1,\Sigma_1)  \}\cup\{ (M_1,\tilde{z}_1,\Sigma_1) \} \in \theforcing{\lambda}{\mu}$ forces that $x = f(\delta^G_\alpha) \in \Hull(M^G_\alpha,\delta_\alpha^G).$
	\end{proof}
\end{proof}

\begin{definition}
For every $\alpha < \omega_1,$ define $\mathscr{M}_\alpha^G$ to be the transitive collapse of $M_\alpha^G.$
\end{definition}

We now put the foregoing together.

\begin{theorem}\label{th it in extension} Suppose  $\NS$ is precipitous.
	If $G$ is a $\theforcing{\lambda}{\mu}$-filter generic over $V,$ then in $V[G]$ there exists a generic iteration  $\iterat = (\mathscr{M}_\beta, j_{\alpha\beta},G_\alpha : \alpha < \beta < \omega_1)$  of length $\omega_1$ of the model $\mathscr{M}_0^G$ with direct limit equal to $H_\lambda^V.$ Moreover, if $H_\lambda^\#$ exists in $V,$ then the model $\mathscr{M}_0^G$ is iterable.
\end{theorem}
\begin{proof}
	This follows from Proposition~\ref{prop it-ind fact} together with Lemma~\ref{lem unique_generic_iteration}. Iterability of $\mathscr{M}_0^G$ follows from Lemma~\ref{lem trans_coll}.
\end{proof}

By combining Theorem~\ref{th it in extension} and Woodin's Lemma~\ref{lem itbounded}, one derives the following conclusions. Note in the first point of Corollary~\ref{cor it in extension} that the $\theforcing{\lambda}{\mu}$-extension is closed under sharps for reals because of Lemma~\ref{lem sharps in extension}. In the third point of Corollary~\ref{cor it in extension}, we reobtain a result of Claverie and Schindler (\cite[Corollary 21]{claverie_schindler}), it  is implied by the second point of Corollary~\ref{cor it in extension} together with Schindler's theorem that  $\mathsf{BMM}$ implies the existence of sharps for all sets (see \cite[Theorem 1.3]{schindler}).

\begin{corollary}\label{cor it in extension} Suppose  $\NS$ is precipitous.
	\begin{enumerate}
		\item If $H_\theta^\#$ exists for $\theta = (2^{<\mu})^+$, then 
		\[\theforcing{\lambda}{\mu} \Vdash   \mathbf{u}_2 > \lambda.\]
		\item 	If $H_\lambda^\#$ exists and the bounded forcing axiom for the forcing $\theforcing{\lambda}{\mu}$ holds,
		\[ \text{ then }\mathbf{u}_2 = \omega_2. \]
		\item \cite[Corollary 21]{claverie_schindler} If Bounded Martin's Maximum $\mathsf{BMM}$ holds, then $\mathbf{u}_2 = \omega_2.$
	\end{enumerate}
\end{corollary}

\begin{remark}
$\;$	
	\begin{enumerate}
		\item By \cite[Theorem 5.7]{doebler_schindler}, the forcing $\theforcing{\lambda}{\mu}$ is semiproper if and only if the principle $(\dagger)$ holds, that is,  $\theforcing{\lambda}{\mu}$ is semiproper if and only if every ssp forcing is semiproper. Also by  \cite[Theorem 5.7]{doebler_schindler}, $(\dagger)$ is equivalent to $\textrm{Glob-}\mathsf{SCC}^{\textrm{cof}}$ which by Proposition~\ref{prop globSCC games} is equivalent to the existence of club many countable $M\prec H_\theta$ for all $\theta \gg \lambda$ such that $P2 \nearrow \mathbf{G}^{c\text{-}c}_M(H_\lambda).$
		\item It is interesting to compare the forcing $\theforcing{\lambda}{\mu}$ with the forcings from \cite{claverie_schindler} and \cite{ketlarzap}; and although all three are very similar, the exact relation between these forcings remains largely open. We note that the model $M_0^G,$ added by  $\theforcing{\lambda}{\mu},$ has the same strong properties as the countable family of functions that occurs on the generic branch of the Namba forcing in \cite{ketlarzap}. Namely:
		\begin{itemize}
			\item every countable elementary submodel of $H_\lambda^V$ that contains $M_0^G$ as a subset is selfgeneric,
			\item $M_0^G$ contains as a subset a countable set of functions $\{g_n : n < \omega\}$ from $\omega_1^{<\omega}$ to $H_\lambda^V$ such that $H_\lambda^V = \bigcup_{n<\omega} \ran(g_n).$

		\end{itemize}
	The second of these two points is already clear from the foregoing, the first one is slightly more surprising and can be checked through a similar density argument.
				\item The proof of \cite[Lemma 2.1]{doebler_schindler} can directly be adapted to show that if  $\NS$ is precipitous, $H_\theta^\#$ exists for $\theta = (2^{<\mu})^+$, and the bounded forcing axiom for the forcing $\theforcing{\lambda}{\mu}$ holds, then also the admissible club guessing principle $\mathsf{ACG}$ holds. $\mathsf{ACG}$ states that the club-filter on $\omega_1$ is generated by sets of the form\linebreak $ C_x = \{ \alpha < \omega_1 : L_\alpha[x] \vDash \mathsf{KP} \}$ with $x \in \mathbb{R}.$ 
	\end{enumerate}
\end{remark}

\begin{acknow} 
	BV is supported by an Emergence en recherche grant from the Universit\'e Paris Cit\'e (IdeX).
	BDB is supported by a Cofund MathInParis PhD fellowship. This project has received funding from the European Union’s Horizon 2020 research and innovation programme under the Marie Sk\l{}odowska-Curie grant agreement No.~754362.~
	\includegraphics[width=0.55cm]{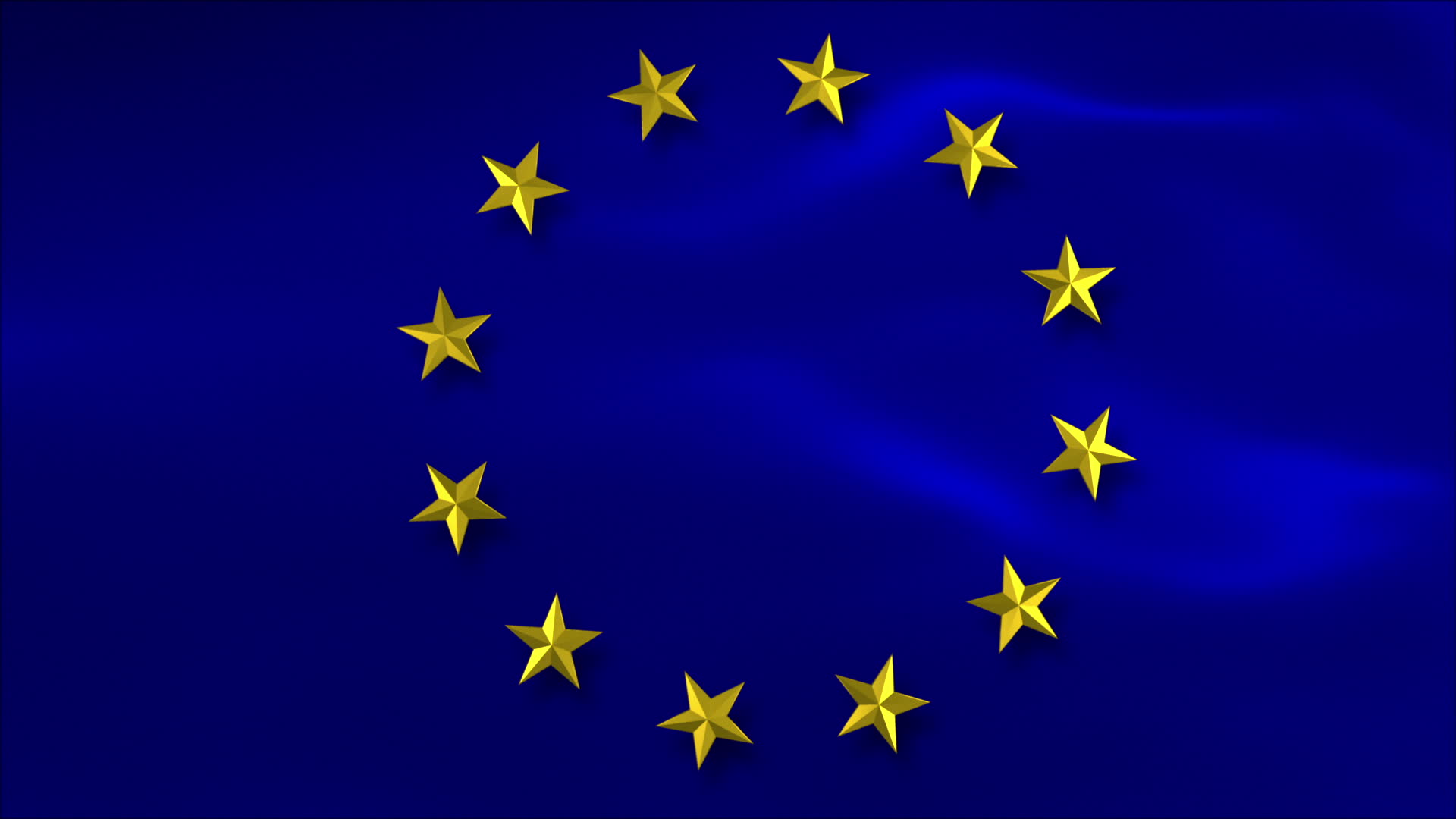}
	
\end{acknow}

\bibliographystyle{amsplain}
\bibliography{BibliographyKLZ}

\providecommand{\bysame}{\leavevmode\hbox to3em{\hrulefill}\thinspace}
\providecommand{\MR}{\relax\ifhmode\unskip\space\fi MR }
\providecommand{\MRhref}[2]{%
  \href{http://www.ams.org/mathscinet-getitem?mr=#1}{#2}
}
\providecommand{\href}[2]{#2}
\begin{thebibliography}{10}

\bibitem{aspero_schindler}
David {Asper{\'o}} and Ralf Schindler, \emph{{Martin's} {{Maximum}$^{++}$}
  implies {Woodin's} axiom {$(\ast)$}}, Ann. Math. \textbf{193} (2021),
  793--835.

\bibitem{claverie_schindler}
Benjamin Claverie and Ralf Schindler, \emph{Increasing {$u_2$} by a stationary
  set preserving forcing}, The Journal of Symbolic Logic \textbf{74} (2009),
  no.~1, 187--200.

\bibitem{cox}
Sean Cox, \emph{Adjoining only the things you want: a survey of {Strong}
  {Chang’s} {Conjecture} and related topics}, To appear in Research Trends in
  Contemporary Logic (2020).

\bibitem{cox_zeman}
Sean Cox and Martin Zeman, \emph{Ideal projections and forcing projections},
  The Journal of Symbolic Logic \textbf{79} (2014), no.~4, 1247--1285.

\bibitem{doebler_schindler}
Philipp Doebler and Ralf Schindler, \emph{{$\Pi_2$} consequences of
  {$\mathsf{BMM} + \mathsf{NS}_{\omega_1}$} is precipitous and the
  semiproperness of all stationary set preserving forcings}, Math. Res. Letters
  \textbf{16} (2009), 797--815.

\bibitem{fengjechzap}
Qi~Feng, Thomas Jech, and Jind\v{r}ich Zapletal, \emph{On the structure of
  stationary sets}, Sci China Ser A \textbf{50} (2007), no.~4, 615--627.

\bibitem{foreman}
Matthew Foreman and Akihiro Kanamori, \emph{Handbook of set theory}, Springer,
  Dordrecht, 2010.

\bibitem{jensen_Lforcing}
Ronald Jensen, \emph{Subcomplete forcing and {$\mathcal{L}$-forcing}},
  {E-recursion}, {Forcing} And {C$^*$-algebras}, Lecture notes series,
  Institute for Mathematical Sciences, National University of Singapore,
  vol.~27, 2012, pp.~83--181.

\bibitem{ketlarzap}
Richard Ketchersid, Paul Larson, and Jind\v{r}ich Zapletal, \emph{Increasing
  {$\delta_2^1$} and {Namba-Style Forcing}}, Journal of Symbolic Logic
  \textbf{72} (2007), no.~4, 1372--1378.

\bibitem{martin}
Donald~A. Martin, \emph{Projective sets and cardinal numbers: some questions
  related to the continuum problem}, Lecture Notes in Logic, p.~484–508,
  Cambridge University Press, 2011.

\bibitem{mitchell}
William~J. Mitchell, \emph{Adding closed unbounded subsets of {$\omega_2$} with
  finite forcing}, Notre Dame Journal of Formal Logic \textbf{46} (2005),
  no.~3, 357--371.

\bibitem{moschovakis}
Yiannis~N. Moschovakis, \emph{Determinacy and prewellorderings of the
  continuum}, Mathematical Logic and Foundations of Set Theory (Yehoshua
  Bar-Hillel, ed.), Studies in Logic and the Foundations of Mathematics,
  vol.~59, Elsevier, 1970, pp.~24--62.

\bibitem{schindler}
Ralf Schindler, \emph{Semi-proper forcing, remarkable cardinals, and {Bounded}
  {Martin's} {Maximum}}, Mathematical Logic Quarterly \textbf{50} (2004),
  no.~6, 527--532.

\bibitem{Schlicht}
Philipp Schlicht, \emph{Thin equivalence relations and inner models}, Annals of
  Pure and Applied Logic \textbf{165} (2014), no.~10, 1577--1625.

\bibitem{steel_welch}
John~R. Steel and Philip~D. Welch, \emph{{$\Sigma^1_3$} absoluteness and the
  second uniform indiscernible}, Israel Journal of Mathematics \textbf{104}
  (1998), 157--190.

\bibitem{steel_wesep}
John~R. Steel and Robert~Van Wesep, \emph{Two consequences of determinacy
  consistent with choice}, Transactions of the American Mathematical Society
  \textbf{272} (1985), 67--85.

\bibitem{todorcevic_proper_forcing_axiom}
Stevo Todorcevic, \emph{A note on the proper forcing axiom}, Axiomatic set
  theory (Providence, Rhode Island), Contemp. Math., vol.~31, AMS, 1984,
  pp.~209--218.

\bibitem{todorcevic_directed_sets}
\bysame, \emph{Directed sets and cofinal types}, Trans. Amer. Math. Soc.
  \textbf{290} (1985), no.~2, 711--723.

\bibitem{todorcevic_partitions_topology}
\bysame, \emph{Partition problems in topology}, Contemporary Mathematics,
  vol.~84, AMS, 1989.

\bibitem{todorcevic_irredundant_sets}
\bysame, \emph{Irredundant sets in {Boolean} algebras}, Trans. Amer. Math. Soc.
  \textbf{339} (1993), no.~1, 35--44.

\bibitem{woodin}
W.~Hugh Woodin, \emph{The axiom of determinacy, forcing axioms, and the
  nonstationary ideal}, De Gruyter, Berlin, New York, 2010.

\end{thebibliography}
\end{document}